\title{On the $\alpha$-spectral radius of unicyclic and bicyclic graphs with a fixed  diameter}
\author{  $^1$, Haiying Shan$^2$\thanks{Corresponding author},  Yuyao Zhai\\[5pt]
{School of Mathematical Sciences, Tongji University, Shanghai, P. R. China} }

\author{
Feifei Wang\thanks{\footnotesize School of Mathematical Sciences, Tongji University, Shanghai 200092, China
(\texttt{1710854@tongji.edu.cn})}~,~~
Haiying Shan\thanks{\footnotesize School of Mathematical Sciences, Tongji University, Shanghai 200092, China
(\texttt{shan\_haiying@tongji.edu.cn})}~,~~
Yuyao Zhai\thanks{\footnotesize School of Mathematical Sciences, Tongji University, Shanghai 200092, China
(\texttt{1930909@tongji.edu.cn})}
}

\date{}

\documentclass{article}
\usepackage{amsthm}
\usepackage{amsmath}
\usepackage{amsfonts,amssymb,amscd}
\usepackage{mathrsfs,mathtools}
\usepackage{hyperref,latexsym}
\usepackage[figurename=Fig.]{caption}
\usepackage{indentfirst}
\usepackage{bm}
\usepackage{booktabs}
\usepackage[table,xcdraw]{xcolor}
\usepackage{multirow}
\usepackage{rotating,tabularx}
\usepackage{subfig}

\usepackage{adjustbox}
\usepackage{indentfirst}
\usepackage{enumerate,longtable, tabu}
\graphicspath{{./img/}{./figures}}
\usepackage{tikz}

\usetikzlibrary{calc}

\newcommand{\case}[2]{\par \noindent \textbf{Case #1.} #2\par}
\newcommand{\subcase}[2]{\par \noindent \textbf{Subase #1.} #2\par}

\newtheorem{thm}{Theorem}[section]
\newtheorem{lem}{Lemma}[section]
\newtheorem{definition}{Definition}[section]

\newtheorem{proof*}{proof}[section]
\newtheorem{cor}{Corollary}[section]
\newtheorem{con}{Conjecture}[section]
\newtheorem{pro}{Proposition}[section]

\DeclareRobustCommand\Inf{%
	\mathop{
		\kern-3pt\vcenter{\hbox{
			\tikz[anchor=base,  baseline=-3pt, xscale=0.8, yscale=1]{
				\draw[line width = 0.4pt] (-0.2em, 0em)--(0.2em, 0em);
					\fill[line width = 2pt] (-0.4em, 0.4em)--(0.4em, -0.4em)--(0.4em, 0.4em)--(-0.4em, -0.4em)--(-0.4em, 0.4em);
				}
		}}\kern-2pt}
}

\DeclareRobustCommand\BTr{%
	\mathop{
		\kern-3pt\vcenter{\hbox{
			\tikz[anchor=base,  baseline=-3pt, xscale=1.4, yscale=1.4]{
					\fill (90:0.4em)--(210:0.4em)--(-30:0.4em)--cycle;
				}
		}}\kern-2.5pt}
}

\DeclareRobustCommand\INF{%
	\kern-3pt\vcenter{\hbox{
			\tikz[anchor=base,  baseline=-3pt, xscale=0.8, yscale=0.8]{
				\draw[line width = 0.8pt] (-0.4em, 0.4em)--(0.4em, -0.4em)--(0.4em, 0.4em)--(-0.4em, -0.4em)--cycle;
			}
	}}\kern-2.5pt
}

\DeclareRobustCommand\THETA{
	\kern-3pt\vcenter{\hbox{
		\tikz[anchor=base,  baseline=-3pt, scale=2.1]{
			\draw[line width = 1.2pt] (0em, 0.2em)--(-0.2em, 0em)--(0em, -0.2em)--(0.2em, 0em)--cycle;
			\draw[line width = 1.2pt] (-0.2em, 0em)--(0.2em, 0em);
		}
	}}\kern-2.5pt
}

\DeclareRobustCommand\ThetA{
	\kern-3pt\vcenter{\hbox{
		\tikz[anchor=base,  baseline=-3pt, scale=2.1]{
			\draw[line width = 0.6pt] (0em, 0.2em)--(-0.2em, 0em)--(0em, -0.2em)--(0.2em, 0em)--cycle;
			\draw[line width = 0.6pt] (-0.2em, 0em)--(0.2em, 0em);
		}
	}}\kern-2.5pt
}

\newcommand{\x}{{\ensuremath{ \mathbf{x}}}}

\newcommand{\diam}{\operatorname{diam}}

\begin{document}

\maketitle
\begin{abstract}
    The $\alpha$-spectral radius of  a connected graph $G$ is the spectral radius of $A_\alpha$-matrix of $G$.  In this paper,  we discuss the methods for comparing $\alpha$-spectral radius of graphs. As applications, we characterize the graphs with the maximal $\alpha$-spectral radius among all unicyclic and bicyclic graphs of order $n$ with diameter $d$, respectively. Finally,  we determine the unique graph with  maximal signless Laplacian spectral radius among bicyclic graphs of order $n$ with diameter $d$. From our conclusion, it is known that the result of Pai and Liu  in [On the signless Laplacian spectral radius of bicyclic graphs with fixed diameter. Ars Combinatoria 2017, 249–265] is wrong.

\noindent{\bfseries Keywords:} Unicyclic graph;  Bicyclic graph; Diameter ; alpha-spectral radius
\noindent{\bfseries AMS Classification:} 05C35;  05C50
\end{abstract}

\section{Introduction}
All graphs considered in this article are finite, undirected and simple. Let $G=(V(G),E(G))$ be a graph with $n$ vertices and $m$ edges (so $n=|V(G)|$ is its order and $m=|E(G)|$ is its size). Let $G$ be a graph with adjacency matrix $A(G)$ and  $D(G)$ be the diagonal matrix of its vertex degrees. In \cite{MR3648656} the matrix $A_{\alpha}(G)$ has been defined for any real $\alpha\in [0,1]$ as
$$A_{\alpha}(G)=\alpha D(G)+(1-\alpha)A(G),$$
which has recently attracted more and more researchers' attention. One reason for this is that the $\alpha$-spectrum seems to be more informative than other commonly used graph matrices.

Note that $A_0(G)=A(G)$ and $2A_{\frac{1}{2}}(G)=Q(G)$. Thus, the family $A_{\alpha}(G)$ extends both $A(G)$ and $Q(G)$.

For square matrix $A$, let $\phi(A,x)=\det (xI-A)$ denote the characteristic polynomial of $A$ and $\rho(A)$ be the  spectral radius of $A$. We write $\phi_{\alpha}(G)=\phi(A_{\alpha}(G), x)$ and $\rho_{\alpha}(G)=\rho(A_{\alpha}(G))$, where $\rho_{\alpha}(G)$ is called the $\alpha$-spectral radius of $G$.

The study of the largest $\alpha$-eigenvalue remains an attractive topic for researchers. In particular, the extremal values of the $\alpha$-spectral radius for various classes of graphs, and corresponding extremal graphs, have been investigated.

Let $G$ be a connected  graph with $n$ vertices and $m$ edges. 	If $m=n+c-1$, then $G$ is called a $c$-cyclic graph. Specially, if $c=0,1$, or $2$, then $G$ is called a tree, an unicyclic graph, a bicyclic graph. Let $G$ be a $c$-cyclic graph. The base of $G$,  denoted by $\widehat{G}$,  is the (unique)  minimal $c$-cyclic subgraph of $G$.  It is easy to see that $\widehat{G}$ can be obtained from $G$ by consecutively deleting pendent edges.

Let $N_G(v)$ denote the neighbor set of vertex $v$ in $G$, then $d_G(v)=|N_G(v)|$ is called the degree of $v$ of $G$. If there is no confusion, we write $N_G(v)$ as $N(v)$, and $d_G(v)=d(v)=d_v$.

The diameter of $G$, denoted $\diam(G)$, is the maximal distance between any two vertices in the graph.

It is an interesting problem concerning graphs with maximal or minimal spectral radii over a given class of graphs. As early as $1985$, Brualdi and Hoffman \cite{MR774347} investigated the maximal spectral radius of the adjacency matrix of a graph in the set of all graphs with a given number of vertices and edges. Their work was followed by other people, in the connected graph case as well as in the general case. In \cite{MR2404214}, Hansen et al. determined graphs with the largest spectral radius among all the graphs on $n$ vertices with given diameter $d$ or radius $r$. Several types of special graphs with diameter $d$ were also discussed. For example, the spectral radii and signless Laplacian spectral radii of trees, unicyclic graphs, bicyclic graphs and tricyclic graphs on $n$ vertices with fixed diameter were discussed in \cite{MR2298999, MR2213058, MR3643499,MR2278222, MR2859926, MR2769349}. Nikiforov \cite{MR3786248} determined the graph with the largest $\alpha$-spectral radius among all graphs with $n$ vertices and diameter at least $d$.

In this paper, motivated by the above results, we will determine the extremal graph with maximal $\alpha$-spectral radius among  all unicyclic or bicyclic graphs with $n$ vertices and diameter $d$, respectively.

\section{Notations and Preliminaries}
Let $G$ be a connected graph with $V(G)=\{v_1,\dots,v_n\}$. A column vector $\boldsymbol{x}=(x_{v_1}, \dots, x_{v_n})^T\in \mathbb{R}^n$ can be considered as a function defined on $V(G)$ which maps vertex $v_i$ to $x_{v_i}$, i.e., $\x(v_i)=x_{v_i}$ for $i=1,\dots,n$. Then $$\boldsymbol{x}^TA_{\alpha}(G)\boldsymbol{x}=\sum_{\{u,v\}\in E(G)}(\alpha(x_u^2+x_v^2)+2(1-\alpha)x_ux_v).$$

If $G$ is connected then $A_{\alpha}(G)$ is irreducible, and by the Perron-Frobenius theory of non-negative matrices, $\rho_{\alpha}(G)$ has multiplicity one and there exists a unique positive unit eigenvector corresponding to $\rho_{\alpha}(G)$. We shall refer to such an eigenvector as the $\alpha$-Perron vector of $G$. If $\boldsymbol{x}$ is the $\alpha$-Perron vector of $G$, then for each $u\in V(G)$,
$$\rho_{\alpha}(G)x_u=\alpha d(u)x_u+(1-\alpha)\sum_{\{u,v\}\in E(G)}x_v.$$
For a unit column vector $\boldsymbol{x}\in \mathbb{R}^n$ with at least one nonnegative entry, by Rayleigh quotient's principle, we have $\rho_{\alpha}(G)\ge \boldsymbol{x}^TA_{\alpha}(G)\boldsymbol{x}$ with equality holds if and only if $\boldsymbol{x}$ is the
$\alpha$-Perron vector of $G$.

\begin{definition}\cite{belardoCombinatorialApproachComputing2010a}
	For any $n \times n$  matrix $A=\left(a_{i j}\right)$,  the Coates digraph associated to $A$, denoted by $G_{A}$,  is a weighted digraph defined as follows:
	\begin{itemize}
		\item the vertex set of $G_{A}$ is equal to $\{1,2, \ldots, n\}$, where the ith vertex  corresponds to the ith row (or equivalently, to the ith column) of $A$;
		\item the arc set of $G_{A}$ consists of all arcs of the form $uv$ with weight
			  $a_{uv}(1 \leqslant u, v \leqslant n) ;$ for $u=v$ the corresponding arc is a loop. If the weight of some arc is zero, then it is ignored.
	\end{itemize}	
\end{definition}

For a Coates graph $G=G_A$, we call the matrix $A$, denoted by $A(G)$,  the  weighted adjacency matrix of $G$.
Let $U$ be a vertex subset of graph $G$ and  $G_\alpha$ be the Coates graph with respect to matrix $A_\alpha(G)$. We write $\psi_{\alpha}(G,U)=\phi(G_\alpha-U,x)$. The spectral radius of $A(G_\alpha-U)$ is denoted by $\rho(G_\alpha-U)$.

In the following lemma, we will list some results on nonnegative matrices and graphs, which  follow directly from
Perron-Frobenius theorem (see \cite{MR1298430}).
\begin{lem}\cite{ MR3648656, MR1298430, MR3786249}\label{lemc2} Let $G$ be a  connected graph. Some basic results on the spectrum are the following:
	\begin{enumerate}[(i).]
		\item   Let $A$ and $B$ be two nonnegative matrices with $A \geq B$ and $A \neq B$. If  there exists a principle submatrix $M$ of $A$ such that $B \leq M$, then $\rho(B)\leq \rho(A)$ and the inequality is strict when $A$ irreducible.
		\item Let $1\ge \alpha \ge \beta \ge 0$. If $G$ is a graph of order $n$ with $A_{\alpha}(G)=A_{\alpha}$ and $A_{\beta}(G)=A_{\beta}$, then $\lambda_k(A_{\alpha})-\lambda_k(A_{\beta})\ge 0$ for any $k\in [n]$.
		\item Let $G$ be a connected graph. If $G$ is not a tree and $G\ncong C_n$, then $\rho_{\alpha}(G)>2$.
	\end{enumerate}
\end{lem}

\begin{lem} \cite{2021arXiv210503077S}\label{lemz2}
	Let $A$ and $B$ be two nonnegative matrices of order $n$ with $A \geq B$ and $A \neq B$. Then the following holds:
	\begin{center}
		$\phi(B,x) \geq \phi(A,x)$ for $x \geq \rho(A)$,
	\end{center}
	especially, when $A$ is irreducible, the inequality is strict.
\end{lem}
\begin{pro}\cite{MR3786248}\label{proN1}
	Let $\alpha\in [0,1)$ and  $G$ be a graph with  $\rho_{\alpha}(G)>2$. Take  $P=u_1u_2\dots u_{r}(r\ge 2)$ be a pendent path in $G$ with root $u_1$ and  $\x$ be the
	$\alpha$-Perron vector of $G$. Then $\x(u_i)>\x(u_j)$ for $1\leq i< j \leq r$.
\end{pro}

\begin{lem}\cite{MR3786249}\label{lem8}
	Let $G$ be a connected graph with $\alpha\in [0,1)$. For $u,v\in V(G)$, suppose $N\subseteq N(v)\backslash (N(u)\cup \{u\})$. Let $G'=G-\{vw:w\in N\}+\{uw:w\in N\}$. If $N\neq \varnothing$ and $\boldsymbol{x}=(x_1,\dots, x_n)$ is the $\alpha$-Perron vector of $G$ such that $x_u\ge x_v$, then $\rho_{\alpha}(G')>\rho_{\alpha}(G)$.
\end{lem}
By Lemma \ref{lem8}, we can easily obtain  the following results.
\begin{lem}\cite{MR3786249}\label{lem6}
	Let $\alpha\in [0,1)$, and $G$ be a connected graph and $e=uv$ be a cut-edge of $G$. Let  $G'$ be the graph obtained from $G$ by deleting the edge $uv$, identifying $u$ with $v$, and adding a pendant edge to $u(=v)$. Then $\rho_{\alpha}(G')>\rho_{\alpha}(G)$.
\end{lem}

Let $G$ be a graph.  Assume that $e_1=uv, e_2=wy$, $e'_1=uw,e'_2=vy$ and $e_1,e_2\in E(G), e'_1,e_2 \notin E(G)$.  Take $G'=G-\{e_1,e_2\}+\{e'_1,e'_2\}$. We say that $G'$ is obtained from $G$ by 2-switching operation $e_1 \xrightleftharpoons[w]{\;v} e_2$.

\begin{lem}\cite{MR3988716}\label{lema1}
	Let $G, G'$ be the graphs defined as above and $\boldsymbol{x}$ be the $\alpha$-Perron vector of $G$ for $\alpha\in [0,1)$. If $x_u\ge x_y$ and $x_w\ge x_v$, then $\rho_{\alpha}(G')\ge \rho_{\alpha}(G)$. Furthermore, if one of two inequalities is strict, then $\rho_{\alpha}(G')> \rho_{\alpha}(G)$.
\end{lem}

An internal path of $G$ is a path $P$ (or cycle) with vertices $v_1,v_2,\dots, v_k$ (or $v_1=v_k$) such that $d_{v_1}\ge 3,\ d(v_k)\ge 3$ and $d(v_2)=\dots=d(v_{k-1})=2$.
\begin{lem} \cite[Lemma 1.1]{MR3988716}\label{lem7}
	Let $G$ be a connected graph with $\alpha\in [0,1)$ and $uv$ be some edge on an internal path of $G$. Let $G_{uv}$ denote the graph obtained from $G$ by subdivision of edge $uv$ into edges $uw$ and $wv$. Then $\rho_{\alpha}(G_{uv})<\rho_{\alpha}(G)$.
\end{lem}
Let $u$, $v$ be two vertices of connected graph $G$ with degree at least $2$.
Take $G_{u,v}(k,l)$ be the graph obtained from $G$ by attaching the pendent paths $P_k$ to $u$ and  $P_l$ to $v$, respectively.
\begin{lem}\cite{MR3786249}\label{lem100}
	Let $G_{u,v}(k,l)$ be as defined above. If $k-l\ge 2$ and $(u,v) \in E(G)$, then  $\rho_{\alpha}(G_{u,v}(k-1,l+1))>\rho_{\alpha}(G_{u,v}(k,l))$ for $\alpha\in [0,1)$.
\end{lem}

In \cite{belardoCombinatorialApproachComputing2010a}, Francesco Belardo et al. present a Schwenk-like formula for symmetric matrices.
\begin{thm}[\cite{belardoCombinatorialApproachComputing2010a}]\label{weightedcharpoly}
	Let $A=(a_{ij})$ be any symmetric square matrix of order $n$, and let $G\left(=G_{A}\right)$ be its Coates digraph with $V(G)=\{1,2,\dots,n\}$. If $v$ is a fixed vertex of $G$ then
	$$
		\phi(G)=\left(x-a_{v v}\right) \phi(G-v)-\sum\limits_{u\neq v}a_{uv}^2\phi(G-u-v)-2\sum_{C \in \mathcal{C}_{v}(G)} \omega_{A}(C) \phi(G-V(C)),
	$$
	where $\mathcal{C}_{v}(G)$ is the set of all undirected cycles of $G$ of length $\geqslant 3$ passing through $v$ and $\omega_{A}(C)=\prod \limits_{(i,j) \in E(C)} a_{i j}$.
\end{thm}

Given two disjoint rooted graphs $G$ and $H$ with roots $u$ and $v$, respectively, then the coalescence of $G$
and $H$ is the graph $G(u,v)H$ obtained by identifying the roots $u$ and $v$. 
The coalescence of  rooted graphs can be naturally extended to weighted digraphs (including weighted graphs). 

\begin{lem} \cite{belardoCombinatorialApproachComputing2010a}\label{lem1}
	Let $G(u,v)H$ be the coalescence of two rooted weighted digraphs (possibly with loops) $G$ and $H$ whose roots are $u$ and $v$, respectively. Then
	\begin{align*}
	\phi(G(u,v)H)=\phi(G)\phi(H-v)+\phi(G-u)\phi(H)-x\phi(G-u)\phi(H-v).
	\end{align*}
	\end{lem}


Given three disjoint rooted Coates (di)graphs $G_1$, $G_2$ and $H$, $u_i$ is the root vertex of $G_i$ for $i=1,2$ and $v_1,v_2$ are two vertices of $H$. The rooted product of $H$ and $G_1,G_2$ with respect to $(u_1,u_2)$ and $(v_1,v_2)$, denoted by $H(G_1,G_2)$,  the Coates (di)graphs obtained from  $G_1$, $G_2$ and $H$   by coalescing  $u_i$ and  $v_{i}$ into new vertices $w_i$ for $i=1,2$.

Repeatedly using Lemma \ref{lem1}, we  have the following result.
\begin{lem}\label{lem:ch1}
	Let $G, H_i$ be connected graphs with $u,v \in V(G), w_i\in V(H_i)$ for $i=1,2.$
	Let $H=G(H_1,H_2)$ be rooted product of $G$ and $(H_1,H_2)$ with respect to $(u,v)$ and $(w_1,w_2)$, respectively, then
	\begin{align*}
		\phi_{\alpha}(H)=gg'_1g'_2+g_u(g_1-xg'_1)g'_2+g_v(g_2-xg'_2)g'_1+g_{uv}(g_1-xg'_1)(g_2-xg'_2),
	\end{align*}

	where $g=\phi_{\alpha}(G)$, $g_u=\psi_{\alpha}(G,u)$, $g_v=\psi_{\alpha}(G, v)$, $g_{uv}=\psi_{\alpha}(G,\{u,v\})$, $g_1=\phi_{\alpha}(H_1)$, $g'_1=\psi_{\alpha}(H_1, w_1)$, $g_2=\phi_{\alpha}(H_2)$ and $g'_2=\psi_{\alpha}(H_2, w_2)$.
\end{lem}
By Lemma \ref{lem:ch1}, we can easily get the following lemma.
\begin{lem}\label{lem:y2}Let $G, H_1, H_2$ be connected graphs with $u \in V(G), \{u_i,v_i\}\subset V(H_i)$ for $i=1,2.$
	Take $G_i=H_i(G,G)$ be rooted product of $H_i$ and $(G,G)$ with respect to vertex sequences $(u_i,v_i)$ and $(u,u)$ for $i=1,2$. Let
$$
\begin{aligned}
	DF_1(x)= & \phi_{\alpha}(H_1)-\phi_{\alpha}(H_2),                                                       \\
	DF_2(x)= & \psi_{\alpha}(H_1,u_1)+\psi_{\alpha}(H_1,v_1)-\psi_{\alpha}(H_2,u_2)-\psi_{\alpha}(H_2,v_2), \\
	DF_3(x)= & \psi_{\alpha}(H_1,\{u_1,v_1\})-\psi_{\alpha}(H_2,\{u_1,v_1\}).
\end{aligned}
$$
	Then
	\begin{align*}
		\phi_{\alpha}(G_1)-\phi_{\alpha}(G_2) & =DF_1(x)\psi_{\alpha}(G,u)^2+DF_2(x)\big(\phi_{\alpha}(G)-x\psi_{\alpha}(G,u)\big)\psi_{\alpha}(G,u) \\
		                                      & +DF_3(x)\big(\phi_{\alpha}(G)-x\psi_{\alpha}(G,u)\big)^2.
	\end{align*}
\end{lem}

\section{Techniques for comparing the $\alpha$-spectral radii of graphs}

In the following, we list some necessary lemmas which will be used in order to prove the main result of this article.
\begin{lem} \label{lem4}
	Let $G, H$ be connected graphs and  $u\in V(G), v_1,v_2\in V(H)$. For $i=1,2$, take $G_i=G(u,v_i)H$ and when $x\ge \rho(H_\alpha-v_1)$, if $\psi_{\alpha}(H,v_2)>\psi_{\alpha}(H,v_1)$ 
	then $\rho_\alpha(G_1) < \rho_\alpha(G_2).$
\end{lem}
\begin{proof}
	By Lemma \ref{lem1}, for $i=1,2$, we have
	$$\phi_{\alpha}(G_i,x)=\psi_{\alpha}(G,u)\phi_{\alpha}(H,x)+(\phi_{\alpha}(G,x)-x\psi_{\alpha}(G,u))\psi_{\alpha}(H,v_i).
	$$
	Thus, 
	\begin{align}\label{eq1}
		\phi_{\alpha}(G_2,x)-\phi_{\alpha}(G_1,x)=(\phi_{\alpha}(G,x)-x\psi_{\alpha}(G,u))(\psi_{\alpha}(H,v_2)-\psi_{\alpha}(H,v_1)).
	\end{align}
	By Lemma \ref{lemz2}, we have  $\phi_{\alpha}(G,x)-x\psi_{\alpha}(G,u)<0$ for all $x\ge \rho_{\alpha}(G)$. By (i) of Lemma \ref{lemc2}, we have $\rho_\alpha(G_2)>\max\{\rho_\alpha(G),\rho(H_\alpha-v_1)\}$.
	From $\psi_{\alpha}(H,v_2)>\psi_{\alpha}(H,v_1)$ for $x\ge \rho(H_\alpha-v_1)$ and Formula \eqref{eq1},  we have $\phi_{\alpha}(G_2,x)<\phi_{\alpha}(G_1,x)$ for $x\ge \rho_\alpha(G_2)$. Then $\rho_\alpha(G_1) < \rho_\alpha(G_2)$ hold.
\end{proof}

\begin{lem} \label{lem:4}
	Let $G$ and $H$ be two graphs,  $w\in V(H)$ and $u_1,u_2\in V(G)$ such that $N_G(u_1)\backslash \{u_2\}\subset N_G(u_2)\backslash \{u_1\}$. Take $H_1=G(u_1, w)H$ and $H_2=G(u_2,w)H$, then we have $\rho_{\alpha}(H_2)>\rho_{\alpha}(H_1)$.
\end{lem}

\begin{proof}
	Let $C^*[u_1]$ be the set of all cycles containing $u_1$ in $G$ with $u_2 \not\in V(C)$ and $C^*[u_2]$ be the set of all cycles containing $u_2$ in $G$ with $u_1 \not\in V(C)$.  Take $N_1=N_G(u_1)\backslash \{u_2\}$ and $N_2=N_G(u_2)\backslash \{u_1\}$.
	From $N_1 \subset N_2$, we know that $d(u_2)>d(u_1)$ and there exists  an injective mapping $\Phi$ from $C^*[u_1]$ to $C^*[u_2]$ such that  $V(C)\cup \{u_2\}=V(\Phi(C))\cup \{u_1\}$  for any $C \in C^*[u_1]$.

	Applying
	Theorem \ref{weightedcharpoly} to $A_\alpha(G)(u_1)$ and $A_\alpha(G)(u_2)$, respectively, we have
	\begin{equation}\label{eq:cochar1}
		\begin{aligned}
			\psi_{\alpha}(G,u_1)= & (x-\alpha d(u_2)) \psi_{\alpha}(G,\{u_1,u_2\})
			-(1-\alpha)^2\sum_{\mathclap{w\in N_2}} \psi_{\alpha}(G,\{u_1,u_2, w\})                                     \\
			                      & -2\sum_{\mathclap{Z\in C^*[u_2]}}(1-\alpha)^{|Z|}\psi_{\alpha}(G,V(Z)\cup \{u_1\}),
		\end{aligned}
	\end{equation}
	\begin{equation}\label{eq:cochar2}
		\begin{aligned}
			\psi_{\alpha}(G,u_2)= & (x-\alpha d(u_1)) \psi_{\alpha}(G,\{u_1,u_2\})
			-(1-\alpha)^2\sum_{\mathclap{w\in N_1}} \psi_{\alpha}(G,\{u_1,u_2, w\})                                     \\
			                      & -2\sum_{\mathclap{Z\in C^*[u_1]}}(1-\alpha)^{|Z|}\psi_{\alpha}(G,V(Z)\cup \{u_2\}).
		\end{aligned}
	\end{equation}
	Let $DN=N_2\backslash N_1$ and $DC=C^*[u_2]\backslash C^*[u_1]$. From Formulae \eqref{eq:cochar1} and \eqref{eq:cochar2}, we have
	\begin{align*}
		df(x)= & \psi_{\alpha}(G,u_2)-\psi_{\alpha}(G,u_1)                                                                                \\
		=      & \alpha (d(u_2)-d(u_1)) \psi_{\alpha}(G,\{u_1,u_2\})+(1-\alpha)^2\sum_{\mathclap{w\in DN}}\psi_{\alpha}(G,\{u_1,u_2, w\}) \\
		       & +2\sum_{Z\in DC}(1-\alpha)^{|Z|}\psi_{\alpha}(G,V(Z)\cup \{u_1\}).
	\end{align*}
	Since  $A(G_\alpha-\{u_1,u_2\}),A(G_\alpha-\{u_1,u_2,w\}),A(G_\alpha-V(Z)\cup\{u_1\})$ are principal submatries of $A(G_\alpha-u_1)$, the spectral radius of these submatries are not more than $\rho(G_{\alpha}-u_1)$ according to (i) of Lemma \ref{lemc2}. Hence,
	when $x>\rho(G_{\alpha}-u_1)$,  $\psi_{\alpha}(G,\{u_1,u_2\})>0$, $\psi_{\alpha}(G,\{u_1,u_2, w\})>0$ and $\psi_{\alpha}(G,V(Z)\cup \{u_1\})>0$. Therefore,  	$\psi_{\alpha}(G,u_2)-\psi_{\alpha}(G,u_1)>0$ for $x>\rho(G_{\alpha}-u_1)$. By Lemma \ref{lem4}, we have $\rho_{\alpha}(H_2)>\rho_{\alpha}(H_1)$.
\end{proof}

Knowing the symmetries of a graph $G$ can be quite useful to find the spectral radius of the graph $G$.
Thus, we say that $u$ and $v$ are equivalent in $G$, if there exists an automorphism $\tau$ of $G$ such that $\tau(u)=v$.
\begin{lem}\label{lem:11}
	Let $H_1,\ H_2$ and $G$ be three graphs, $w_1\in V(H_1), \ w_2\in V(H_2)$ and $u,v\in V(G)$ such that $u$ and $v$ be equivalent vertices in $G$. Let $H=G(H_1,H_2)$ be rooted graph of $(H_1,H_2)$ and $G$ with respect to vertex sequences $(w_1,w_2)$ and $(u,v)$, respectively,  then
$$\psi_{\alpha}(H,u)-\psi_{\alpha}(H,v)=g_3(F'_1F_2-F_1F'_2),$$
where $F_1=\phi_{\alpha}(H_1)$, $F'_1=\psi_{\alpha}(H_1,w_1)$, $F_2=\phi_{\alpha}(H_2)$, $F'_2=\psi_{\alpha}(H_2,w_2)$ and $g_3=\psi_{\alpha}(G,\{u,v\})$.
\end{lem}
\begin{proof}
Let $g=\phi_{\alpha}(G)$, $g_1=\psi_{\alpha}(G,u), g_2=\psi_{\alpha}(G,v)$, then by Lemma \ref{lem1}, we have
\begin{align*}
	\psi_{\alpha}(H,u)= & \psi_{\alpha}(H_1,w_1)\big(\psi_{\alpha}(G,u)\psi_{\alpha}(H_2,w_2)+\psi_{\alpha}(G,\{u,v\})\phi_{\alpha}(H_2) \\
						& -x\psi_{\alpha}(G,\{u,v\})\psi_{\alpha}(H_2,w_2)\big )                                                         \\
	=                   & F'_1(g_2F'_2+g_3F_2-xg_3F'_2).
\end{align*}
Similarly, we have $$\psi_{\alpha}(H,v)=g_1F'_1F'_2+g_3F_1F'_2-xg_3F'_1F'_2.$$
Since $u$ and $v$ are equivalent vertices in $G$, we have $g_1=g_2$. Thus,
\begin{center}
	\hspace{0.25\linewidth}$\psi_{\alpha}(H,u)-\psi_{\alpha}(H,v)=g_3(F'_1F_2-F_1F'_2).$ \qedhere
\end{center}
\end{proof}
Let $u$ be the pendent vertex of path $P_t$. In the sequel, we always denote $f_{t-1}(x)=\psi_{\alpha}(P_{t},u)$. If there is no confusion, we write $f_{t-1}(x)$ as $f_{t-1}$.
It is easy to see that $f_0(x)=1$, $f_1(x)=x-\alpha$ and
\begin{equation}\label{rec:path}
	\begin{aligned}
		 & \phi_{\alpha}(P_t)=f_t(x)+\alpha f_{t-1}(x),         \\
		 & f_t(x)=(x-2\alpha)f_{t-1}(x)-(1-\alpha)^2f_{t-2}(x).
	\end{aligned}
\end{equation}

Then for $l-k=p \geq 0$,  we have
\begin{equation}\label{eq:recurrence}
	\begin{aligned}
		  & \phi_{\alpha}(P_{k+1})f_l(x)- \phi_{\alpha}(P_{l+1})f_k(x)=\begin{vmatrix}
			\phi_{\alpha}(P_{k+1}) & f_k(x) \\
			\phi_{\alpha}(P_{l+1}) & f_l(x)
		\end{vmatrix} \\
		= & \begin{vmatrix}
			f_{k+1}(x) & f_k(x) \\
			f_{l+1}(x) & f_l(x)
		\end{vmatrix}
		=(1-\alpha)^2\begin{vmatrix}
			f_{k}(x) & f_{k-1}(x) \\
			f_{l}(x) & f_{l-1}(x)
		\end{vmatrix}=\dots                                             \\
		= & (1-\alpha)^{2k}((x-\alpha)f_{p}(x)-f_{p+1}(x))                                        \\
		= & (1-\alpha)^{2k}(\alpha f_{p}(x)+(1-\alpha)^2f_{p-1}(x))\text{  (for $p \geq 1$)}.
	\end{aligned}
\end{equation}
So we have the following conclusion:
\begin{lem}\label{lem:pathrec}
	Suppose $l,k$ are positive number with $l \geq k$. Then, for $x \geq 2$,
	$$
		\begin{aligned}
			  & \phi_{\alpha}(P_{k+1})f_l(x)- \phi_{\alpha}(P_{l+1})f_k(x) 
			=  f_{k+1}(x)f_l(x)-f_k(x)f_{l+1}(x) \geq 0,
		\end{aligned}
	$$ and the equality holds if and only if $l=k$.
\end{lem}
\begin{lem}\label{lem3}
	Let $u$ and $v$ be two equivalent vertices of connected graph $G$, $w$ be a vertex of connected graph $W$ and $H=G_{u,v}(k,l)$. Take $H_1=H(u,w)W$ and $H_2=H(v,w)W$, if $l-k\ge 1$, then $\rho_{\alpha}(H_2)>\rho_ {\alpha}(H_1)$.
\end{lem}
\begin{proof}
	By Lemma \ref{lem:11} and Formula \eqref{eq:recurrence}, we have
	\begin{align*}
		\psi_{\alpha}(H,v)-\psi_{\alpha}(H,u) & =\psi_{\alpha}(G,\{u,v\})\big (\phi_{\alpha}(P_{k+1})f_l(x)- f_k(x)\phi_{\alpha}(P_{l+1})\big) \\
		                                      & =\psi_{\alpha}(G,\{u,v\})\big (f_{k+1}(x)f_l(x)-f_k(x)f_{l+1}(x)\big ).
	\end{align*}
	When $l-k\ge 1$, by Lemma \ref{lem:pathrec}, we have $f_{k+1}(x)f_l(x)-f_k(x)f_{l+1}(x)>0$ for $x\ge 2$.  Since $\rho(H_\alpha-u)>\rho(G_\alpha-\{u,v\})$,
	we have
	\begin{center}
		$\psi_{\alpha}(H,v)-\psi_{\alpha}(H,u)>0$  for $x>\rho(H_\alpha-u)$.
	\end{center}
	By  Lemma \ref{lem4}, we have  $\rho_{\alpha}(H_2)>\rho_{\alpha}(H_1)$.
\end{proof}
By Lemma \ref{lem3}, we can easily get the following corollary.
\begin{cor}\label{cor:1}
	Let $u$ and $v$ be two vertices of the connected graph $G$ such that $N_G(u)\backslash \{v\} = N_G(v)\backslash \{u\}$, $H_3$ denote the graph obtained from $G_{u,v}(k,l)$ by identifying vertex $u$ with the center of $K_{1,s}$ and  $H_4$ denote the graph obtained from $G_{u,v}(k,l)$ by identifying vertex $v$ with the center of $K_{1,s}$. If $l-k\ge 1$, then $\rho_{\alpha}(H_4)>\rho_{\alpha}(H_3)$.
\end{cor}

\section{The extremal graphs with the maximal $\alpha$-spectral radius among $ \mathscr{U}(n,d)$}

Let $\mathscr{C}_t(n, d)$ denote the set of all the $n$-vertex $t$-cyclic graph with diameter $d$.
The sets  $\mathscr{C}_1(n, d)$ and $\mathscr{C}_2(n, d)$ also are written as $\mathscr{U}(n, d)$ and $\mathscr{B}(n, d)$, respectively.

\begin{pro}\label{pro:extremal}
	\label{pro3}Let $G$ be the graph  with maximal $\alpha$-spectral radius among $\mathscr{C}_t(n, d)$. Then $G$ has the following properties:
	\begin{enumerate}[(1).]
		\item If $e$ is an edge in some internal path of $G$, then $e$ must be in a $C_3$ of $G$.
		\item For any path $P$ of length $d$ in $G$, $V(P) \cap V(\widehat{G})\neq \varnothing$ always holds.
		\item Let $P$ be a path  of length $d$ in $G$ and $U=V(P)\backslash V(\widehat{G})$. When $U \neq
			      \varnothing$, take $U'=V(G)\backslash (V(P)\cup V(\widehat{G}))$. Then there exists some vertex $w$ in $V(\widehat{G})$ such that $G[U'\cup \{w\}]$ is a star $S$ with center $w$.
		\item Take $G$ and $H$ be graphs with maximal $\alpha$-spectral radius among $\mathscr{U}(n, d)$ and $\mathscr{B}(n, d)$, respectively. Then we have $\widehat{G}=C_3$ and $\widehat{H}\in \{\INF, \ThetA\}$. The notations of $\INF$ and  $\ThetA $   are introduced in Section $5$  of this paper.
	\end{enumerate}
\end{pro}

\begin{proof}
	(1).
	Assume $e$ is an edge in some internal path of $G$ and there exists no $C_3$ in $G$ such that $e \in E(C_3)$.
	Let $G'$ be the graph obtained from $G$ by contracting edge $e$.  By Lemma \ref{lem7}, we have $\rho_{\alpha}(G)<\rho_{\alpha}(G')$. It is easy to see that $0\leq \diam(G)-\diam(G')\leq 1$. When $\diam(G)=\diam(G')+1$, take $G^*$ be a graph obtained from $G'$ by attaching a pendant edge to one end vertex of some path of length $d-1$ in $G'$.  When $\diam(G)=\diam(G')$, suppose that $u \in V(G')$ is not end vertex of any path of length $d$ in $G'$.
	Take $G^*$ be a graph obtained from $G'$ by attaching a pendant edge to vertex $u$. We have $G^*\in \mathscr{C}_t(n, d)$. By $(i)$ of Lemma \ref{lemc2}, $\rho_{\alpha}(G')<\rho_{\alpha}(G^*)$. So $\rho_{\alpha}(G)<\rho_{\alpha}(G^*)$, a contradiction. Hence, $e$ must be in a $C_3$ of $G$.

	(2). Assume that $P$ is path of length $d$ in $G$ such that $V(P_{d+1})\cap V(\widehat{G})=\varnothing$. Then $P_k=u_1u_2\dots u_k$ ($k \geq 2$) be the shortest path between $P$ and $\widehat{G}$. Applying Lemma \ref{lem6} to the edge $u_1u_2$, we get a graph $G^*\in \mathscr{C}_t(n, d)$ such that $\rho_{\alpha}(G)<\rho_{\alpha}(G^*)$, a contradiction. Hence, $V(P_{d+1})\cap V(\widehat{G})\neq \varnothing$.

	(3). Let $G'=G[V(P)\cup V(\widehat{G})]$ and $G^*=G-E(G')$.
	Assume that there exists two connected components, denoted by $T_1, T_2$,  with order greater than one  of $G^*$. By Lemma \ref{lem6}, $T_1$ and $T_2$ must be star graphs with center vertices, denoted by $u,v$, in $V(G')$.
	Take $G_1$ (or $G_2$) be the graphs obtained from $G$ by removing tree $T_1$ (or $T_2$) from $u$ to $v$ (or $v$ to $u$).
	By Lemma \ref{lem8}, we have $\rho_{\alpha}(G)<\rho_{\alpha}(G_1)$ or $\rho_{\alpha}(G)<\rho_{\alpha}(G_2)$, where $G_1,G_2\in \mathscr{C}_t(n, d)$, a contradiction.
	So there  exists at most one connected component, denoted by $T$, with order greater than one in $G^*$. By Lemma \ref{lem7} and  $(i)$ of Lemma \ref{lemc2},  $|V(T)\cap V(\widehat{G})|=1$ holds.  Then $T$ is the desired star graph $S$.

	(4). It immediately follows from (1) of this proposition.
\end{proof}


\begin{figure}[ht]
	\centering
	\subfloat[$\Delta_1(s;a,b)$]{\includegraphics[page=13,width=0.4\textwidth]{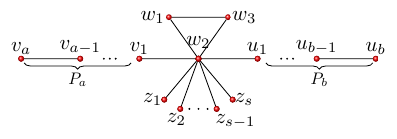}}\qquad
	\subfloat[$\Delta_2(s;a,b)$]{\includegraphics[page=3, width=0.4\textwidth]{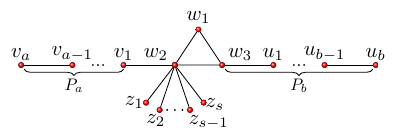}}
	\caption{$\Delta_i(s;a,b)$ for $i=1,2$}\label{fig:unicyclic}
\end{figure}

For $G\in \mathscr{U}(n, d) $, we have $n\ge 3$ and $1\le d\le n-2$. If $d=1$, then $G\cong C_3$. Therefore, in the following, we assume that $d\ge 2$ and $n\ge 4$.

Suppose that $V(C_3)=\{w_1,w_2,w_3\}$, 	We construct some unicyclic graphs obtained from $C_3$ as follows (see  Fig. \ref{fig:unicyclic}):
\begin{itemize}
	\item $\Delta_1(s;a,b)$: the graph obtained from $C_3$ by attaching $s$ pendant edges, pendant path of length $a$ and pendant path of length $b$ to vertex $w_2$.
	\item $\Delta_2(s;a,b)$: the graph obtained from $C_3$ by attaching $s$ pendant edges and pendant path of length $a$ to vertex $w_2$ and attaching pendant path of length $b$ to vertex $w_3$, respectively.
\end{itemize}
Let $\BTr_i(n,d)$ denote the set consisting of all graphs $\Delta_i(s;a,b)$ with $n$ vertices and diameter $d$ for $i=1,2$.
Take $U^*_1(n,d)=\Delta_1(s;a,b) \in \BTr_1(n, d)$ if $0\leq a-b\leq 1$ and $d=a+b$.
Take $U^*_2(n,d)=\Delta_2(s;a,b) \in \BTr_2(n, d)$ if $0\leq a-b\leq 1$ and $d=a+b+1$.

The authors of \cite{MR2278222}  and \cite{MR2859926} determined that $U_2^*(n,d)$ is the unique graph  with maximal adjacency spectral radius and signless Laplacian spectral radius, respectively, among  $\mathscr{U}(n, d)$.

In the following part of this section,  we will extend above results to $\alpha$-spectral graph theory.

\begin{lem}\label{lem:unicyclic}
	For fixed $s,\ a$ and $b$, if $G\in \BTr_1(n,d)$, then there exists some graph $G'\in \BTr_2(n,d)$
	such that $\rho_{\alpha}(G)<\rho_{\alpha}(G')$.
\end{lem}
\begin{proof}
	When  $G\in \BTr_1(n,d)$, Wlog, let $G=\Delta_1(s;a,b)$. Let $H_1$ be the  pendant path of length $b-1$ and $H_2$ the graph obtained from $C_3$ by attaching $s+1$ pendant edges $w_2z_i(i=1,\dots,s+1)$ and pendant path of length $a$ to vertex $w_2$. Let $u$ be a pendant vertex of $H_1$, then $\Delta_1(s;a,b)=H_1(u,z_{s+1})H_2$ and $G'=\Delta_2(s+1;a,b-1)=H_1(u,w_3)H_2\in \BTr_2(n,d)$. Clearly, $N_{H_2}(z_{s+1})\subset N_{H_2 }(w_3)$, then by Lemma \ref{lem:4}, we have $\rho_{\alpha}(G)=\rho_{\alpha}(\Delta_1(s;a,b))<\rho_{\alpha}(\Delta_2(s+1;a,b-1))=\rho_{\alpha}(G')$.
\end{proof}

\begin{thm}\label{thml1}
	The graph $U_2^*(n,d)$ is the unique graph with  the maximal $\alpha$-spectral radius  among $\mathscr{U}(n, d)$.
\end{thm}
\begin{proof}
	Let $G$ be the graph with the maximal $\alpha$-spectral radius among $ \mathscr{U}(n, d)$.
	By Proposition \ref{pro:extremal}, we have $G\in \BTr_1(n,d)$ or $G\in \BTr_2(n,d)$.
	According to Lemmas \ref{lem100} and \ref{lem:unicyclic}, we have $G\cong U_2^*(n,d)$.
\end{proof}

\section{The extremal graphs with the maximal $\alpha$-spectral radius among $ \mathscr{B}(n,d)$}

It is known that there are two types of bases of bicyclic graphs,   $\infty(n_1, n_2, n_3)$ is obtained from two vertex-disjoint cycles  $C_{n_1}$ and $C_{n_2}$ by joining vertex $u$ of $C_{n_1}$ and vertex $v$ of $C_{n_2}$ by a path $P_{n_3} (n_3 \geq 1)$,  $\theta(n_1, n_2, n_3)$  ($n_1\geq n_2 \geq n_3$) is obtained from three pairwise internal disjoint paths of lengths $n_1, n_2, n_3$ from vertices $u$ to $v$ (see  Fig.\ref{fig:bicyclic}). The graphs  $\infty(n_1, n_2, n_3)$ and $\theta(n_1, n_2, n_3)$  are referred to as $\infty$-graph and $\Theta$-graph, respectively.

Let
$$\begin{aligned}
	\mathscr{B}_1(n,d) & =\{G| G\in \mathscr{B}(n,d) \text{ and $\widehat{G}$ is $\infty$-graph}\}, \\
	\mathscr{B}_2(n,d) & =\{G| G\in \mathscr{B}(n,d) \text{ and $\widehat{G}$ is $\Theta$-graph}\}.
\end{aligned}$$

\begin{figure}[ht]
	\centering
	\subfloat[$\infty(n_1,n_2,n_3)$]{\includegraphics[page=1,width=0.4\textwidth]{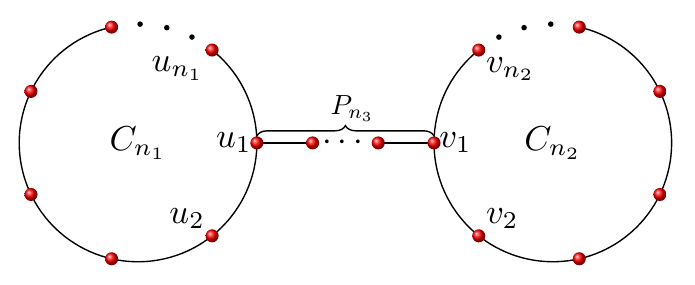}}\qquad
	\subfloat[$\Theta(n_1,n_2,n_3)$]{\includegraphics[page=2, width=0.4\textwidth,trim=0 0 0.5cm  0cm ]{figure}}
	\caption{$\infty(n_1,n_2,n_3)$ and $\Theta(n_1,n_2,n_3)$}\label{fig:bicyclic}
\end{figure}

\begin{figure}[ht]
	\centering
	\subfloat[$\ThetA_1(s;a,b)$]{\includegraphics[page=4,width=0.4\textwidth]{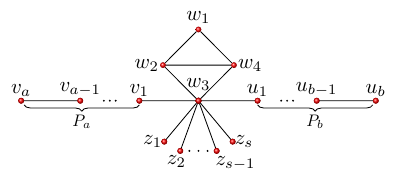}}
	\subfloat[$\ThetA_2(s;a,b)$]{\includegraphics[page=5, width=0.4\textwidth]{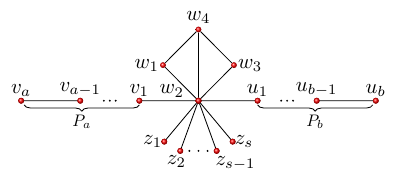}}\\
	\subfloat[$\ThetA_3(s;a,b)$]{\includegraphics[page=6, width=0.4\textwidth]{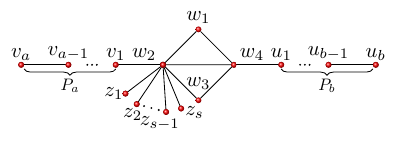}}
	\subfloat[$\ThetA_4(s;a,b)$]{\includegraphics[page=7, width=0.4\textwidth]{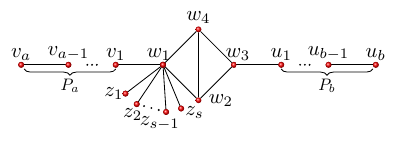}}\\
	\subfloat[$\ThetA_5(s;a,b)$]{\includegraphics[page=8, width=0.4\textwidth]{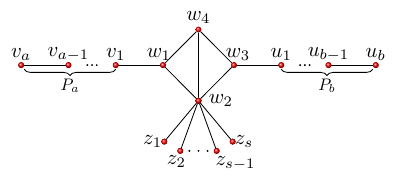}}
	\caption{$\ThetA_i(s;a,b)$ for $i=1,2,3,4,5$}\label{fig:figs}
\end{figure}

Let $\INF$ be the  bicyclic graph obtained  by coalescence two copies of the cycle graph $C_3$  with a common vertex. Suppose that $V(\INF)=\{w_0,w_1,w_2,w_3,w_4\}$ with $d(w_0)=4$ and $\{(w_1,w_2),(w_3,w_4)\}\subset E(\INF)$.

Let $\ThetA$ be the unique bicyclic graph of order 4 with $V(\ThetA)=\{w_1,w_2,w_3,w_4\}$ and $d(w_1)=d(w_3)=2,d(w_2)=d(w_4)=3$.
We construct some bicyclic graphs obtained from $\ThetA$ as follows (see (a)--(e) of Fig. \ref{fig:figs}):
\begin{itemize}
	\item $\ThetA_1(s;a,b)$: the graph obtained from $\ThetA$ by attaching $s$ pendant edges and two pendant paths of lengths $a$ and $b$ to vertex $w_1$.
	\item $\ThetA_2(s;a,b)$: the graph obtained from $\ThetA$ by attaching $s$ pendant edges and two pendant paths of lengths $a$ and $b$ to vertex $w_2$.
	\item $\ThetA_3(s;a,b)$: the graph obtained from $\ThetA$ by attaching $s$ pendant edges and pendant path of length $a$ to vertex $w_2$ and attaching pendant path of length $b$ to $w_4$.
	\item $\ThetA_4(s;a,b)$: the graph obtained from $\ThetA$ by attaching $s$ pendant edges and pendant path of length $a$ to vertex $w_1$ and attaching pendant path of length $b$ to $w_3$.
	\item $\ThetA_5(s;a,b)$: the graph obtained from $\ThetA$ by attaching $s$ pendant edges to $w_2$ and attaching pendant paths of lengths $a$ and $b$ to vertex $w_1$ and $w_3$, respectively.
\end{itemize}
Let $\THETA_i(n, d)$ denote the set consisting of all graphs $\ThetA_i(s;a,b)$  with $n$ vertices and diameter $d$ for $i=1,\dots,5$.
Take $\THETA(n, d)=\bigcup_{i=1}^{3}\THETA_i(n, d)$.

By Lemma  \ref{lem8}, the following result can be deduced.
\begin{lem}\label{lem:Inf2Theta}
	For any graph $G \in \mathscr{B}_1(n,d)$ with $\widehat{G}\cong \INF$, there exists some graph $H \in \mathscr{B}_2(n,d)$ with $\widehat{H}\cong \ThetA$ such that $\rho_\alpha(G)<\rho_\alpha(H)$.
\end{lem}
\begin{proof}
	Suppose that $G \in \mathscr{B}_1(n,d)$ with $\widehat{G}=\INF$ and $V(\INF)=\{w_0,w_1,w_2,w_3,w_4\}$ with $d(w_0)=4$ and $\{(w_1,w_2),(w_3,w_4)\}\subset E(\INF)$. Take $\boldsymbol{x}$ be the $\alpha$-Perron vector of $G$. Wlog, suppose $x_{w_1} \geq x_{w_3}$. Let $H'=G-w_3w_4+w_1w_4$. By Lemma \ref{lem8}, we have $\rho_\alpha(G)<\rho_\alpha(H')$.
	It is easy to see that $0\leq \diam(H')-\diam(G)\leq 1$. If $\diam(H')=d$, take $H=H'$.  If $\diam(H')=d+1$, then $\diam(H')\geq 4$. Therefore, there exist some non-pendant edge $e=uv$ in the path of length $d+1$ such that $e \notin E(\widehat{H'})$. Let $H$ be the graph obtained from $H'$ by contracting edge $e $ into new vertex $w$ and attaching a pendant edge to $w$. By Lemma \ref{lem8}, we have $\rho_{\alpha}(H)>\rho_{\alpha}(H')$ and 
	$\diam(H)=d$.  Since $\widehat{H}\cong \ThetA$, $H$ is the desired graph.
\end{proof}
Take $B^*_3(n,d)=\ThetA_3(s;a,b) \in \THETA_3(n, d)$ if $0\leq a-b\leq 1$ and $d=a+b+1$.
Take $B^*_i(n,d)=\ThetA_i(s;a,b) \in \THETA_i(n, d)$ if $0\leq a-b\leq 1$ and $d=a+b+2$ for $i=4,5$.
In order to complete the proof of our main result, we need the following results.

\begin{lem}\label{thm:s1}
	The graph $B_3^*(n,d)$ is the unique graph with  the maximal $\alpha$-spectral radius  among $\THETA(n, d)$.
\end{lem}
\begin{proof}
	\noindent\textbf{(1)}. For fixed $s,\ a$ and $b$, take $G\cong \ThetA$ with $V(G)=\{w_1,w_2,w_3,w_4\}$ and $d(w_1)=d(w_3)=2,d(w_2)=d(w_4)=3$ and $H$ be the graph obtained  by attaching $s$ pendant edges and two pendant paths of lengths $a$ and $b$ to an isolated vertex $v$.  Then  $\ThetA_1(s;a,b)=G(w_3,v)H$ and $\ThetA_2(s;a,b)=G(w_2,v)H$. Since $N_{G}(w_3)\backslash \{w_2\}\subset N_{G}(w_2)\backslash \{w_3\}$, by Lemma \ref{lem:4}, we have  $$\rho_{\alpha}(\ThetA_1(s;a,b))<\rho_{\alpha}(\ThetA_2(s;a,b)).$$
	\noindent\textbf{(2)}. For fixed $s,\ a$ and $b$, let graph $G\cong \ThetA_2(s;a,1)$ and $H=P_{b}$ with pendant vertex $u$. Then  $\ThetA_2(s;a,b) \cong G(u_1,u)H$ and $\ThetA_3(s+1;a,b-1) \cong G(w_4,u)H$. Clearly, when $b=1$, $\ThetA_2(s;a,1)\cong \ThetA_3(s+1;a,0)$. Since $N_{G}(u_1)\subset N_{G }(w_4)$,  by Lemma \ref{lem:4}, for $2\le b$,   we have $$\rho_{\alpha}(\ThetA_2(s;a,b))<\rho_{\alpha}(\ThetA_3(s+1;a,b-1)).$$
	If $\ThetA_2(s;a,b)\in \THETA(n, d)$, then $\ThetA_3(s+1;a,b-1)$ is also in $\THETA(n, d)$.

	From above arguments in (1) and (2), we have that the graph with the maximal $\alpha$-spectral radius among $\THETA(n, d)$ must be in $\THETA_3(n, d)$. Then by Lemma \ref{lem100} and Corollary \ref{cor:1}, we have the graph $B_3^*(n,d)$ is the unique graph with  the maximal $\alpha$-spectral radius  among $\THETA_3(n, d)$.
	This completes the proof.
\end{proof}

\begin{lem}\label{lema3}When $d \geq 4$,
	for any $H_1=\ThetA_4(s;a,b)\in \THETA_4(n,d)$, there always exists  some graph $H_2\in \THETA_5(n,d)$ such that $\rho_{\alpha}(H_1)<\rho_{\alpha}(H_2)$.
\end{lem}

\begin{proof}
	Let $H=\ThetA_4(0;a,b)=\ThetA_5(0;a,b)$ and $G=K_{1,s}$ with $d_G(u)=s$.
	Then $H_1=\ThetA_4(s;a,b)=H(w_1,u)G$ and $H_2=\ThetA_5(s;a,b)=H(w_2,u)G$.
	Take
	\begin{align*}
		h_1 & =\psi_{\alpha}(H, \{w_1, w_2\}), h_2=\psi_{\alpha}(H, \{w_1, w_2, w_3\}),         \\
		h_3 & =\psi_{\alpha}(H, \{w_1, w_2, w_4\}), h_4=\psi_{\alpha}(H, \{w_1, w_2,w_3,w_4\}),
	\end{align*}
	and 		$$h_5= \begin{cases}
			\psi_{\alpha}(H, \{w_1, w_2, v_1\}) & \text{ when $a \geq 1$}, \\
			0                                   & \text{otherwise}.
		\end{cases}$$
	Applying
	Theorem \ref{weightedcharpoly} to $A(H_\alpha-w_1)$ and $A(H_\alpha-w_2)$, respectively, we have
	\begin{align*}
		\psi_{\alpha}(H, w_1) & =(x-3\alpha)h_1-(1-\alpha)^2 (h_2+h_3) -2(1-\alpha)^3h_4, \\
		\psi_{\alpha}(H, w_2) & =
		\begin{cases}
			(x-3\alpha)h_1-(1-\alpha)^2(h_3+h_5) & \text{when $a \geq 1$}, \\
			(x-2\alpha)h_1-(1-\alpha)^2h_3 & \text{otherwise}.
		\end{cases}
	\end{align*}
	Then,
	\begin{align*}
	\psi_{\alpha}(H, w_2)-\psi_{\alpha}(H, w_1)&=
    \begin{cases}
	(1-\alpha)^2\big ( h_2-h_5\big)+2(1-\alpha)^3h_4& \text{when $a \geq 1$}, \\
	\alpha h_1+(1-\alpha)^2 h_2+2(1-\alpha)^3h_4 & \text{otherwise}.
     \end{cases}
     \end{align*}
	It is easy to see that when $x>\rho(H_\alpha-w_1)$,  $h_1, h_2,  h_4$ are positive. \\
	So $\psi_{\alpha}(H, w_2)>\psi_{\alpha}(H, w_1)$ for $a=0$.

	Now consider the case for $a \geq 1$.	By direct calculation,  we have
	$$\begin{aligned}
			h_2= & (x-3\alpha)f_{a}(x)f_{b}(x),\qquad  h_4=f_{a}(x)f_{b}(x),                               \\
			h_5= & f_{a-1}(x) \psi_{\alpha}(\ThetA_4(0;0,b), \{w_1,w_2\})                                  \\
			=    & f_{a-1}(x) \left((x-3\alpha)f_{b+1}(x)-(\alpha(x-3\alpha)+(1-\alpha)^2)f_{b}(x)\right).
		\end{aligned}
	$$
	Then, we have $\psi_{\alpha}(H, w_2)-\psi_{\alpha}(H, w_1)=(1-\alpha)^2 f^* +2(1-\alpha)^3f_a(x)f_b(x)$, where\\
	$f^*=(x-3\alpha)(f_a(x)f_b(x)-f_{a-1}(x)f_{b+1}(x))+(\alpha(x-3\alpha)+(1-\alpha)^2)f_b(x)f_{a-1}(x).$

	When $a\le b+1$, by Lemma \ref{lem:pathrec}, we have $f_a(x)f_b(x)-f_{a-1}(x)f_{b+1}(x) \ge 0$ for $x\geq 2$, and it is easy to see that when $x>\rho(H_{\alpha}- w_1)$,  $f_b(x), f_{a-1}(x)$ are positive. Then $\psi_{\alpha}(H, w_2)-\psi_{\alpha}(H, w_1)>0$ for $x\ge \rho(H_\alpha-w_1)$. Then by Lemma \ref{lem4},
	we have $\rho_{\alpha}(\ThetA_4(s;a,b))<\rho_{\alpha}(\ThetA_5(s;a,b))$.
	$\ThetA_5(s;a,b)$ is the desired graph for $a\le b+1$.

	When $a\ge b+2$, let $\boldsymbol{x}$ be the $\alpha$-Perron vector of $H_1$ and $v_0=w_1$, $u_0=w_3$, we distinguish the following cases.
	\case {1} {$x_{u_b}\ge x_{v_{a-1}}.$}
	Let $H'_1=H_1-v_{a-1}v_{a}+u_bv_a=\ThetA_4(s;a-1,b+1)$, by Lemma \ref{lem8}, we have $\rho_{\alpha}(\ThetA_4(s;a-1,b+1))=\rho_{\alpha}(H'_1)>\rho_{\alpha}(H_1)=\rho_{\alpha}(\ThetA_4(s;a,b))$.  Then by this and case $a\le b+1$, we have  $\rho_{\alpha}(\ThetA_4(s;a,b))\le \rho_{\alpha}(B^*_4(n,d)) <\rho_{\alpha}(B^*_5(n,d))$.
	$B^*_5(n,d)$ is the desired graph for this case.

	\case{2}{$x_{v_{a-1}}>x_{u_b}.$}
	\subcase{2.1}{$x_{v_{a-i}}>x_{u_{b-i+1}}$ for $1<i\leq b$.}
	When $i=b+1$, we have $x_{v_{a-b-1}}>x_{u_0}$, and by Proposition \ref{proN1}, we have $x_{v_1}>x_{v_{a-b-1}}>x_{u_0}$, let $H'_1=H_1-w_4u_0+w_4v_1=\ThetA_5(s;a-1,b+1)$, by Lemma \ref{lem8},
	we have $\rho_{\alpha}(\ThetA_5(s;a-1,b+1))=\rho_{\alpha}(H'_1)>\rho_{\alpha}(H_1)=\rho_{\alpha}(\ThetA_4(s;a,b))$.
	$\ThetA_5(s;a-1,b+1)$ is the desired graph for this subcase.

	\subcase {2.2}{$x_{v_{a-j+1}}>x_{u_{b-j+2}}$ and $x_{v_{a-j}}\le x_{u_{b-j+1}}$ for some $1<j\leq b$.} 
	Let $e_1=v_{a-j+1}v_{a-j}, e_2=u_{b-j+1}u_{b-j+2}$ and $H'$ be the graph obtained from $H$  
	by 2-switching operation $e_1 \xrightleftharpoons[u_{b-j+2}]{\;v_{a-j}} e_2$. By Lemma \ref{lema1}, we have $$\rho_{\alpha}(\ThetA_4(s;a-1,b+1))=\rho_{\alpha}(H'_1)>\rho_{\alpha}(H_1)=\rho_{\alpha}(\ThetA_4(s;a,b)).$$ Then by this and case $a\le b+1$, we have  $\rho_{\alpha}(\ThetA_4(s;a,b))\le \rho_{\alpha}(B^*_4(n,d)) <\rho_{\alpha}(B^*_5(n,d))$.
\end{proof}
\begin{lem}\label{lem15}
	For $a\ge b\ge 1$, then we have $$\rho_{\alpha}(\ThetA_5(s;a,b))> \rho_{\alpha}(\ThetA_5(s;a+1,b-1)).$$
\end{lem}
\begin{proof}
	Let $\boldsymbol{x}$ be the $\alpha$-Perron vector of $\ThetA_5(s;a+1,b-1)$ and $v_0=w_1$, $u_0=w_3$. Assume that $\rho_{\alpha}(\ThetA_5(s;a+1,b-1))\ge \rho_{\alpha}(\ThetA_5(s;a,b))$. Then the following  assertion holds:

	\noindent{\bfseries{Assertion 1:}} $x_{v_{a-i}}>x_{u_{b-i-1}}$ for all $i=0,\dots,b-1$.

	We prove the assertion by induction on $i$. If $x_{u_{b-1}}\ge x_{v_a}$, then for $H=\ThetA_5(s;a+1,b-1)-v_av_{a+1}+u_{b-1}v_{a+1}$, we  have $H\cong \ThetA_5(s;a,b)$, and thus by  Lemma \ref{lem8}, we have
	$\rho_{\alpha}(\ThetA_5(s;a,b))=\rho_{\alpha}(H)>\rho_{\alpha}(\ThetA_5(s;a+1,b-1))$, a contradiction. Thus, $x_{v_a}>x_{u_{b-1}}$. The assertion holds  for $i=0$. If $b=1$, then $i=0$ and the claim follows.
	Suppose that $b\ge 2$, and $x_{v_{a-i}}>x_{u_{b-i-1}}$, where $0\le i\le b-2$. 
	If $x_{u_{b-(i+1)-1}}\ge x_{v_{a-(i+1)}}$, let $e_1=v_{a-(i+1)}v_{a-i}, e_2=u_{b-(i+1)-1}u_{b-i-1}$ and $H'$ be the graph obtained from $\ThetA_5(s;a+1,b-1)$  by 2-switching operation $e_1 \xrightleftharpoons[v_{a-i}]{\;u_{b-i-1}} e_2$. we have $H'\cong \ThetA_5(s;a,b)$ and thus by Lemma \ref{lema1}, we have $$\rho_{\alpha}(\ThetA_5(s;a,b))=\rho_{\alpha}(H')>\rho_{\alpha}(\ThetA_5(s;a+1,b-1)),$$ which contradicts the assumption. Thus $x_{v_{a-(i+1)}}>x_{u_{b-(i+1)-1}}$. So, Assertion 1 follows by induction.

	By Assertion 1 for $i=b-1$, we have $x_{v_{a-(b-1)}}>x_{w_3}$, and by Proposition \ref{proN1}, we have $x_{v_{1}}>x_{v_{a-(b-1)}}>x_{w_3}$, let $H_1=\ThetA_5(s;a+1,b-1)-w_3w_4+v_1w_4$, we have $H_1\cong \ThetA_4(s;b,a)$ and thus by Lemma \ref{lem8}, we have $\rho_{\alpha}(\ThetA_5(s;a+1,b-1))<\rho_{\alpha}(H_1)=\rho_{\alpha}(\ThetA_4(s;b,a) )$. For $b\le a$,
	by Lemma \ref{lema3}, we have $\rho_{\alpha}(\ThetA_4(s;b,a) )<\rho_{\alpha}(\ThetA_5(s;b,a) )=\rho_{\alpha}(\ThetA_5(s;a,b) )$. Then $\rho_{\alpha}(\ThetA_5(s;a+1,b-1))<\rho_{\alpha}(\ThetA_5(s;a,b) )$,  a contradiction. Thus, we have $\rho_{\alpha}(\ThetA_5(s;a,b))> \rho_{\alpha}(\ThetA_5(s;a+1,b-1)) $.
\end{proof}

In the following, we will determine the graph with the maximal  $\alpha$-spectral radius among  $\mathscr{B}(n,d)$.
\begin{thm}\label{thm1}
	Let $G$ be the graph with the maximal $\alpha$-spectral radius among $\mathscr{B}(n,d)$, then we have
	$G\cong B_3^*(n,d)$ or $G\cong B_5^*(n,d)$.
\end{thm}
\begin{proof}
	By (4) of Proposition \ref{pro:extremal}, Lemma \ref{lem:Inf2Theta}, we have $\widehat{G}=\ThetA$. Take
	$P$ be a path of $G$ with length $d$.
	\case{1} {$|V(P)\cap V(\ThetA)|=1$.}
	Applying Proposition \ref{pro:extremal} and Lemma \ref{lem8}, we can prove that $G\in \THETA_1(n, d)\cup \THETA_2(n, d)$. By Lemma \ref{thm:s1}, we have a graph $H\in \THETA_3(n, d)$ such that $\rho_{\alpha}(H)>\rho_{\alpha}(G)$,
	a contradiction.

	\case{2} {$|V(P)\cap V(\ThetA)|=2$.}
	We have $|E(P)\cup E(\ThetA)|=1$. $G\in \THETA_3(n, d)$ can be deduced from the proof of (3) of Proposition \ref{pro:extremal} and Lemma \ref{lem8}.
	By Lemma \ref{thm:s1}, we have $G\cong B_3^*(n,d)$.

	\case{3} {$|V(P_{d+1})\cap V(\ThetA)|=3$.}
	Applying Proposition \ref{pro:extremal}, we have $G\in \THETA_4(n, d)\cup \THETA_5(n, d)$. By Lemmas \ref{lema3} and \ref{lem15}, we have  further $G\cong B_5^*(n,d)$.

	Combining Cases 1--3, we have $G\cong B_3^*(n,d)$ or $G\cong B_5^*(n,d)$.
	This completes the proof.
\end{proof}
\begin{figure}[ht]
	\centering
	\subfloat[$G_1$]{\includegraphics[page=17,width=0.25\textwidth]{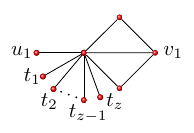}}
	\subfloat[$G_2$]{\includegraphics[page=16,width=0.2\textwidth]{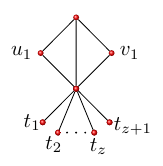}}
	\subfloat[$G_3$]{\includegraphics[page=14,width=0.3\textwidth]{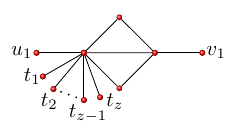}}
	\subfloat[$G_4$]{\includegraphics[page=15,width=0.25\textwidth]{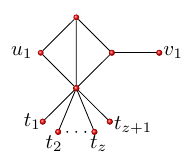}}
	\caption{$G_i$ for $i=1,2,3,4$}\label{fig:figGi}
\end{figure}


Take $n=16, d=9$ and $DR(\alpha)=\rho_{\alpha}(B_3^*(n,d))-\rho_{\alpha}(B_5^*(n,d))$.
By direct calculation, we obtain the value of $DR(\alpha)$ for $\alpha=0, 0.1,\cdots,0.8$.

\begin{table}
	\begin{adjustbox}{width=\textwidth,center}
		\tabcolsep=2pt 
		\extrarowsep=1mm
		\begin{tabu} to 1.3\textwidth {|X[0.8,$c]*9{|X[$c,m]}|}\tabucline -
		\alpha &	0 &	0.1	&	0.2	&	0.3	&	0.4	&	0.5	&	0.6	&	0.7	&	0.8	\\\tabucline -
		DR(\alpha) &	-0.00353	&	-0.0016	&	0.00053	&	0.00237	&	0.00327	&	0.00302	&	0.00207	&	0.00108	&	0.00042	\\\tabucline -
		\end{tabu}
\end{adjustbox}
\caption{ The differences of $\rho_{\alpha}(B_3^*(16,9))$ and $\rho_{\alpha}(B_5^*(16,9))$ for $\alpha=0, 0.1,\cdots,0.8$. }\label{Tab1}
\end{table}
From Table \ref{Tab1}, it can be seen that
the relative magnitudes of the $\alpha$-spectral radii of $B_3^*(n,d) $ and $B_5^*(n,d)$ depends upon the value of $\alpha$. 
	
	In \cite{MR2298999}, the graph $B_3^*(n,3)$ and $B_5^*(n,d)$ are denoted by $P^{\theta}_{4}(3)$  and $P^{+}_{d+1}(\lfloor \frac{d+2}{2} \rfloor)$, respectively. The following result is proved.
\begin{thm}[\cite{MR2298999}]Let $H$ be the graph with the maximal adjacency spectral radius among $\mathscr{B}(n,d)$, then 	 $H\cong B_3^*(n,d)$ for  $d=3$ and $H\cong B_5^*(n,d)$ for  $d \geq 4$. 
\end{thm}
In \cite[Theorem 5.1]{pai:thesis} and \cite[Theorem 3.1]{MR3643499}, the authors asserted that the above result also hold for signless Laplacian spectral radius. Unfortunately, this result is not correct.  By numerical calculation, one can find that: For $n=16$ and $d=9$ and $\alpha=\frac{1}{2}$, 
$\rho_{\alpha}(B_3^*(16,9)) \approx 4.6201$  and  $\rho_{\alpha}(B_5^*(16,9)) \approx 4.6171$. Since $\rho_Q(G)=2\rho_{1/2}(G)$ for any graph $G$, we have $\rho_{Q}(B_3^*(n,d)) > \rho_{Q}(B_5^*(n,d))$, which does not agree with the assertion of Theorem 3.1 of \cite{MR3643499}.

In the following theorem,  we determine the unique graph with the maximal $\alpha$-spectral radius among $\mathscr{B}(n,d)$.
\begin{thm}\label{thm:a1}Let $H$ be the graph with the maximal  signless Laplacian spectral radius among $\mathscr{B}(n,d)$, then 
	$H\cong B_3^*(n,d)$ for  $d\geq 3$.
   \end{thm}
   \begin{proof}
	   For $i=1,2,3,4$, take $G_i$ be graphs depicted in Fig. \ref{fig:figGi}. \\
	   Let $H_i(a,b)=G_i(P_a,P_b)$ be the graph obtained from $G_i$ by attaching pendant paths $P_a$ and $P_b$ to $u_1$ and $v_1$, respectively.
	   Then
	   \begin{align*}
		   H_1(a,b) & \cong \ThetA_3(z;a,b-1),     &
		   H_2(a,b) & \cong \ThetA_5(z+1;a-1,b-1),   \\
		   H_3(a,b) & \cong \ThetA_3(z;a,b),       &
		   H_4(a,b) & \cong \ThetA_5(z+1;a-1,b).
	   \end{align*}
	   Notice that $H_i(1,1)=G_i$ for $i=1,2,3,4$.
	   Take $a=b=l\geq 1$.   We have  for   $d=2l$ and $n=2l+z+3$,
	   $H_1(a,b)  \in \THETA_3(n,d),  H_2(a,b)  \in \THETA_5(n,d)$
	   and  for   $d=2l+1$ and $n=2l+z+4$,
	   $H_3(a,b)  \in \THETA_3(n,d), H_4(a,b)  \in \THETA_5(n,d).$
	   By Lemmas   \ref{thm:s1} and \ref{lem15}, we know that
	   $$B_3^*(n,d)=   \begin{cases}
			   H_1(l,l), & \text{\hspace{-2mm} if }d=2l;    \\
			   H_3(l,l), & \text{\hspace{-2mm}  if }d=2l+1;
		   \end{cases}\text{ \quad and   \quad} B_5^*(n,d)=   \begin{cases}
			   H_2(l,l), & \text{\hspace{-2mm}  if }d=2l;    \\
			   H_4(l,l), & \text{\hspace{-2mm}  if }d=2l+1.
		   \end{cases}$$
   Since $K_{1,z+4}$ is a proper subgraph of $G_i$ for $i=1,2,3,4$, when $\alpha=\frac{1}{2}$,
   we have $$\rho_{\frac{1}{2}}(H_i(l,l))>\rho_{\frac{1}{2}}(K_{1,z+4})= \frac{z+5}{2}.$$

   According to Theorem \ref{thm1}, to complete the proof, it suffices to show that\\ 
   $\rho_{\frac{1}{2}}(B_3^*(n,d))>\rho_{\frac{1}{2}}(B_5^*(n,d))$ for $d\geq 3$.
   
   Thus, it is sufficient to show that 
   $\phi_{\frac{1}{2}}(B_3^*(n,d))<\phi_{\frac{1}{2}}(B_5^*(n,d))$ for $d\geq 3$.

	   For $l \geq 2$, take $dp_l(x)=\phi_{\alpha}(P_l)-xf_{l-1}$. By Lemma \ref{lem:y2}, for $l\geq 2$, we have
	   \begin{align}
		   DF(x)= & \phi_\alpha(B_3^*(n,d))-\phi_\alpha(B_5^*(n,d))\nonumber                \\
		   =      & DF_1(x)f_{l-1}^2(x)+DF_2(x)f_{l-1}(x)dp_l(x)+ DF_3 dp_l(x)^2. \label{difchap}
	   \end{align}

	   For $i=1,2,3,4$, denote
	   $F_{i1}=\phi_{\alpha}(G_i),  F_{i2}=\psi_{\alpha}(G_i,u_1) $, $F_{i3}=\psi_{\alpha}(G_i,v_1) , $\\
	     $ F_{i4}=\psi_{\alpha}(G_i,\{u_1,v_1\}) $.
	   Applying the technique of quotient matrix (see e.g., Corollary 1 of \cite{saravananGeneralizationFiedlerLemma2021}),
	   $F_{ij}$ for $i,j=1,2,3,4$ can be obtained. 
	   
	   Take $f_{ij}$ are polynomials which are listed in Appendix A. Then for $1\leq j \leq 4$, we have
	   $$\begin{aligned}
		   &F_{1j}=(x-\frac{1}{2})^{z-1} f_{1j}, &
	   F_{2j}=&(x-\frac{1}{2})^{z-1} f_{2j},\\
	   &F_{3j}=(x-\frac{1}{2})^{z-1} (x-1)f_{3j},&
	   F_{4j}=&(x-\frac{1}{2})^{z}f_{4j}.\end{aligned}$$

	   \case {1} {$d\equiv 0\pmod {2}$.}
	   Setting $d=2l$, we have $DF(x)=H_1(l,l)-H_2(l,l)$.
	   Since $G_1 \cong G_2$, we have $DF_1(x)=F_{11}-F_{21}=0$. 
	   According to Table \ref{table:charpolys}, by directly calculation,  we have
	   $$\begin{aligned}
		   &DF_2(x)=\frac{1}{2}(x^2 - \frac{z+5}{2}x + 1)(x - \frac{1}{2})^{z} \text{   and   }&
	   DF_3(x)=	\frac{1}{4}(x^2 - \frac{z+5}{2}x + 1)(x - \frac{1}{2})^{z-1}.
	   \end{aligned}
	   $$
	   Then, by Formula \eqref{difchap}, we have
	   $DF(x)=\left((2x-1)f_{l-1}(x)+dp_{l}(x)\right)dp_{l}(x)DF_3(x).$\\
	   Since $\rho_{\frac{1}{2}}(H_2(l,l))>\frac{z+5}{2}>\rho_{\frac{z+5}{2}}(P_{i})(i\ge 2)$, we have $f_{i}(x)>0(i\ge 2)$\\
	    when $x \geq \rho_{\frac{1}{2}}(H_2(l,l))$.
   It is easy to see that\\	
		   $(2x-1)f_{l-1}(x)+dp_{l}(x)=((x-1)f_{l-1}(x)+\phi_{\frac{1}{2}}(P_l))>0$ when $x>2$.\\[2pt]
	   Since $DF_3(x)>0$ and $dp_{l}<0$ when $x>\frac{z+5}{2}$, we have\\
   $DF(x)=H_1(l,l)-H_2(l,l))<0$ when $x \geq \rho_{\frac{1}{2}}(H_2(l,l))$.  \\
	   Thus $\rho_{\frac{1}{2}}(H_1(l,l))>\rho_{\frac{1}{2}}(H_2(l,l))$ follows for this case.

   \case {2} {$d\equiv 1\pmod {2}$.}
   Setting $d=2l+1$, we have $DF(x)=H_3(l,l)-H_4(l,l)$.\\
		   According to Table \ref{table:charpolys}, by directly calculation,  we have
   \begin{align*}
	   DF_1(x)=& -\frac{1}{4}x(x^2-\frac{z+5}{2}x+1)(x-\frac{1}{2})^{z+1},\\
	   DF_2(x)=& 
   -\frac{1}{2}(x^2-\frac{z+5}{2}x+1)(x^{3} - 2x^{2} + \frac{3}{2}x - \frac{1}{4})(x-\frac{1}{2})^{z-1},\\
   DF_3(x)=& -\frac{1}{2}(x^2-\frac{z+5}{2}x+1)(x-\frac{1}{2})^{z-1}(x - 1)^{2}.
   \end{align*}
   Take $g(x)=-\frac{1}{8}(x^2-\frac{z+5}{2}x+1)(x-\frac{1}{2})^{z-1}.$ \\
   For $\alpha=\frac{1}{2}$, 
   since $dp_l(x)=-\frac{1}{2}f_{l-1}(x)-\frac{1}{4}f_{l-2}(x), f_{l-1}(x)=(x-1)f_{l-2}(x)-\frac{1}{4}f_{l-3}(x)$ for $l\geq 3$ and \\ $g(x)<0, f_{l-2}(x)>f_{l-3}(x)$,    $2 x^{3} - \frac{9}{2} x^{2} + \frac{5}{2} x - \frac{1}{4}>\frac{3}{4} x^{2} -\frac{9}{8} x + \frac{3}{8}$ hold when $x>\frac{z+5}{2}$,\\ by Formula \eqref{difchap}, we have
   $$\begin{aligned}
	   DF(x)=&g(x)f_{l-1}(x)\pmb{\big(}3(x-1)(x-\frac{1}{2})f_{l-1}(x)-(x^{3}-3x^{2}+\frac{7}{2}x-\frac{5}{4})f_{l-2}(x)\pmb{\big)}\\
	   &+\frac{1}{4}(x-1)^2f_{l-2}^{2}(x)g(x)\\
	   <&g(x)f_{l-1}(x)\pmb{\big(}3(x-1)(x-\frac{1}{2})f_{l-1}(x)-(x^{3}-3x^{2}+\frac{7}{2}x-\frac{5}{4})f_{l-2}(x)\pmb{\big)}\\
	   =&g(x)f_{l-1}(x)\pmb{\big(}(2x^{3}-\frac{9}{2}x^{2}+\frac{5}{2}x-\frac{1}{4})f_{l-2}(x)-(\frac{3}{4}x^{2}-\frac{9}{8}x+\frac{3}{8})f_{l-3}(x)\pmb{\big)}<0
   \end{aligned}$$

   holds, when $x \geq \rho_{\frac{1}{2}}(H_2(l,l))$ and $l\geq 3$.
   
   When $l=2$ and $x>\frac{z+5}{2}$, the following holds:
   $$\begin{aligned}
   &3 (x - 1) (x - \frac{1}{2}) f_{l-1}(x)- (x^{3} - 3 x^{2} + \frac{7}{2} x - \frac{5}{4})f_{l-2}(x)
   =2x^3-3x^2+\frac{1}{4}x+\frac{1}{2}>0.
   \end{aligned}$$
   Then $DF(x)<0$ is also true  when $x \geq \rho_{\frac{1}{2}}(H_2(l,l))$ and $l=2$.
	   Thus, $\rho_{\frac{1}{2}}(H_3(l,l))>\rho_{\frac{1}{2}}(H_4(l,l))$ follows for this case.
	   By Theorem \ref{thm1}, we have $B_3^*(n,d)$ with the maximal signless Laplacian spectral radius among $\mathscr{B}(n,d)$. This completes the proof.
   \end{proof}

   Motivated by Theorems \ref{thm1}, \ref{thm:a1} and results of numerical calculation,   
   we conclude this paper with the following conjecture.
   \begin{con}
	   When $\frac{1}{2}<\alpha<1$, let $G$ be the graph with the maximal $\alpha$-spectral radius among $\mathscr{B}(n,d)$, then we have
	   $G\cong B_3^*(n,d)$.
   \end{con}

\section*{Conflict of interest}

The authors declare that they have no conflict of interest.

   \newpage
   \appendix
   \section{The expressions of $f_{ij}$ for $i,j=1,2,3,4$}
   \begin{table}[ht]
	\caption{The expressions of $f_{ij}$ for $i,j=1,2,3,4$.}\label{table:charpolys}%
	\begin{adjustbox}{width=1\textwidth,rotate=0,center,scale=0.95}
		\tabulinesep =1mm
		\begin{tabu} to 1.05\textwidth {|X[1,l]|}\tabucline -
$
\begin{aligned}
 f_{11}=f_{12}=&\pmb{\big(}x^{3} - \frac{1}{2}(z+9) x^{2} + (z + 5) x - 1\pmb{\big)}(x - \frac{1}{2})^2(x - 1))\\
 f_{12}=&\pmb{\big(}x^{3} - \frac{1}{2}(z+9) x^{2} + (z + \frac{21}{4}) x - \frac{3}{2}\pmb{\big)}(x - \frac{1}{2})(x - 1)\\
 f_{13}=&\pmb{\big(}x^{3} - \frac{1}{2}(z+7) x^{2} + \frac{1}{2} (z + \frac{11}{2}) x - \frac{1}{2}\pmb{\big)}(x - \frac{1}{2})(x - 1)\\
 f_{14}=&\pmb{\big(}x^{3} -\frac{1}{2}(z+7)x^{2} + \frac{1}{2}( z + 6) x - \frac{3}{4}\pmb{\big)}(x - 1)\\
 \end{aligned}$\\\tabucline -
$\begin{aligned}
 f_{22}=f_{23}=&\pmb{\big(}x^{4} - (\frac{1}{2} z + 5) x^{3} + (\frac{5}{4} z + \frac{31}{4}) x^{2} - (\frac{5}{8} z + \frac{35}{8}) x + \frac{3}{4}\pmb{\big)}(x - \frac{1}{2})\\
 f_{24}=&\pmb{\big(}x^{3} - (\frac{1}{2} z + 4) x^{2} + (\frac{3}{4} z + \frac{17}{4}) x - 1\pmb{\big)}(x - \frac{1}{2})   \end{aligned}$  \\\tabucline -
$\begin{aligned}
 f_{31}=&\pmb{\big(}x^{5} - \frac{1}{2}( z + 12) x^{4} + (\frac{7}{4} z + \frac{23}{2}) x^{3} - (\frac{11}{8} z + \frac{17}{2}) x^{2} + (\frac{1}{4} z + \frac{5}{2}) x - \frac{1}{4}\pmb{\big)}(x - \frac{1}{2}) \\
f_{32}=&x^{5} - \frac{1}{2}( z + 12) x^{4} + (\frac{7}{4} z + \frac{47}{4}) x^{3} - (\frac{11}{8} z + \frac{75}{8}) x^{2} + (\frac{1}{4} z + \frac{51}{16}) x - \frac{3}{8}\\
f_{33}=&\pmb{\big(}x^{4} -\frac{1}{2} ( z + 11) x^{3} + (\frac{3}{2} z + 9) x^{2} - (\frac{3}{4} z + \frac{39}{8}) x + \frac{3}{4}\pmb{\big)}(x - \frac{1}{2})\\
f_{34}=&x^{4} - \frac{1}{2}( z + 11) x^{3} + (\frac{3}{2} z + \frac{37}{4}) x^{2} - (\frac{3}{4} z + \frac{45}{8}) x + \frac{9}{8}	\end{aligned}$ \\\tabucline -
 $\begin{aligned}
	 f_{41}=&x^{6} - (\frac{1}{2} z + 7) x^{5} + (\frac{9}{4} z + \frac{71}{4}) x^{4} - (\frac{13}{4} z + \frac{83}{4}) x^{3} + (\frac{27}{16} z + \frac{185}{16}) x^{2} -(\frac{1}{4} z + \frac{23}{8}) x + \frac{1}{4}\\
	 f_{42}=&x^{5} - (\frac{1}{2} z + 6) x^{4} + (\frac{7}{4} z + \frac{49}{4}) x^{3} - (\frac{13}{8} z + \frac{83}{8}) x^{2} + (\frac{5}{16} z + \frac{55}{16}) x - \frac{3}{8}\\
	 f_{43}=&x^{5} - (\frac{1}{2} z + \frac{13}{2}) x^{4} + (2 z + \frac{59}{4}) x^{3} - (\frac{19}{8} z + \frac{117}{8}) x^{2} + (\frac{13}{16} z + \frac{99}{16}) x - \frac{7}{8}\\
	 f_{44}=&\pmb{\big(}x^{3} - (\frac{1}{2} z + \frac{9}{2}) x^{2} + (z + \frac{21}{4}) x - \frac{5}{4}\pmb{\big)}(x - 1)\\
	 \end{aligned}$\\\tabucline -
		\end{tabu}
	\end{adjustbox}
\end{table}   
   \end{document}